\documentclass[11pt]{amsart}
\usepackage{amsmath,amsthm, amssymb, amsfonts, amscd}
\usepackage{verbatim}

\usepackage[all]{xy}
\usepackage{amsxtra}
\usepackage{url}
\usepackage{enumerate}
\DeclareFontEncoding{OT2}{}{} 
\DeclareTextFontCommand{\textcyr}{\fontencoding{OT2}
    \fontfamily{wncyr}\fontseries{m}\fontshape{n}\selectfont}

\usepackage{comment}

\newtheorem{lemma}{Lemma}

\newtheorem{corollary}{Corollary}

\newtheorem*{theorem*}{Theorem}
\newtheorem*{proposition*}{Proposition}
\newtheorem*{lemma*}{Lemma}
\newtheorem*{corollary*}{Corollary}
\newtheorem*{question*}{Question}
\newtheorem*{conjecture*}{Conjecture}
\newtheorem*{claim*}{Claim}

\newtheorem*{introtheorem*}{Theorem}
\newtheorem*{introproposition*}{Proposition}
\newtheorem*{introlemma*}{Lemma}
\newtheorem*{introcorollary*}{Corollary}

\theoremstyle{definition}

\theoremstyle{remark}

\numberwithin{equation}{section}

\newcommand{\Gal}{\operatorname{Gal}}

\newcommand{\GL}{\operatorname{GL}}

\newcommand{\End}{\operatorname{End}}

\newcommand{\Z}{\mathbb{Z}}

\newcommand{\Q}{\mathbb{Q}}
\newcommand{\R}{\mathbb{R}}
\newcommand{\C}{\mathbb{C}}

\def\AA{{\mathbf{A}}}

\DeclareFontEncoding{OT2}{}{} 
\DeclareTextFontCommand{\textcyr}{\fontencoding{OT2}
    \fontfamily{wncyr}\fontseries{m}\fontshape{n}\selectfont}

\def\C{{\mathbb{C}}}
\def\Q{{\mathbb{Q}}}
\def\Z{{\mathbb{Z}}}

\def\im{{\rm im}}

\def\GL{{\rm GL}}

\renewcommand{\Gal}{\textup{Gal}}

\renewcommand{\GL}{{\rm GL}}

\newcommand{\beq}{\begin{equation}}
\newcommand{\eeq}{\end{equation}}

\renewcommand{\End}{\textup{End}}

\def\Hg{{\rm Hg\,}}
\def\sZ{{\mathcal{Z}}}
\def\AA{{\mathfrak{A}}}
\def\nc{{\rm nc}}
\def\Qbar{{\overline{\Q}}}

\title[Hodge group]
{
The Hodge group and endomorphism algebra\\ of an Abelian variety
}

\author{Mikhail Borovoi}

\thanks{M.V. Borovoi, The Hodge group and the algebra of endomorphisms of an Abelian variety
(Russian). In:  ``Problems in Group Theory and Homological Algebra'', pp. 124--126,
Yaroslav. Gos. Univ., Yaroslavl, 1981.}

\begin{document}


\begin{abstract}
This is an English translation of the author's 1981 note in Russian, published
in a Yaroslavl collection.
We prove that if an Abelian variety over  $\C$  has no nontrivial endomorphisms,
then its Hodge group is $\Q$-simple.
\end{abstract}

\maketitle

In this note we prove that if an Abelian variety $A$ over $\C$ has no nontrivial endomorphisms,
then its Hodge group $\Hg A$ is a $\Q$-simple algebraic group.
Actually a slightly more general result is obtained.
The note was inspired by Tankeev's paper \cite{Tankeev}.
The author is very grateful to Yu.G. Zarhin for useful discussions.
\medskip

Let $A$ be an Abelian variety over $\C$.
Set $V=H_1(A,\Q)$.
Denote by $T^1$ the compact one-dimensional torus over $\R$:
$T^1=\{z\in\C : |z|=1\}$.
Denote by $\varphi\colon T^1\to \GL(V_\R)$ the homomorphism
defining the complex structure in $V_\R=H_1(A,\R)$.
By definition, the Hodge group $\Hg A$ is the smallest  algebraic subgroup
$H\subset \GL(V)$ defined over $\Q$ such that $H_\R\supset \im\, \varphi$.
Denote by $\End\,A$ the ring of endomorphisms of $A$,
and set $\End^\circ A=\End\, A\otimes_\Z \Q$.

\begin{theorem*}
Let $A$ be a polarized Abelian variety.
Let $F$ denote the center of $\End^\circ A$, and let $F_0$ denote the subalgebra
of fixed points in $F$ of the Rosati involution induced by the polarization.
Set $G=\Hg A$ and denote by $r$ the number of factors in the decomposition
of the commutator subgroup $G'=(G,G)$ of $G$ into an almost direct product of $\Q$-simple groups.
Then $r\le \dim_\Q F_0$\,.
\end{theorem*}

\begin{corollary}\label{cor:1}
If $F=\Q$ (in particular, if $\End^\circ A=\Q$), then $\Hg A$ is a $\Q$-simple group.
\end{corollary}

Before proving the Theorem and deducing Corollary 1, we describe the necessary properties of the Hodge group.

\begin{proposition*}[see \cite{Mumford-Families}]
Let $A$ be an Abelian variety over $\C$. Then
\begin{enumerate}[\upshape(a)]
\item The Hodge group $G=\Hg A$ is a connected reductive group.
\item The centralizer $K$ of $\im\,\varphi$ in $G_\R$ is a maximal compact subgroup of $G_\R$\,.
\item $G'_\R$ is a group of Hermitian type (i.e., its symmetric space admits a structure
of a Hermitian symmetric space).
\item The algebra $\End^\circ A$ is the centralizer of $\Hg A$ in $\End\, V$.
\item For any polarization $P$ of $A$, the Hodge group $\Hg A$ respects
the corresponding nondegenerate skew-symmetric form $\psi_P$ on the space $V$.
\end{enumerate}
\end{proposition*}

\noindent
{\em Deduction of Corollary \ref{cor:1} from the Theorem.} Denote by $C$ the center of $G$.
From assertion (d) of the Proposition, it follows that $C\subset F^*$.
Hence, under the hypotheses of the corollary we have $C\subset \Q^*$.
From assertion (b) of the Proposition it follows that $C_\R$ is a compact group,
hence, $C$ is a finite group and $G=G'$. By virtue of the Theorem, $r=1$ and $G$ is a $\Q$-simple group.

\begin{lemma}\label{lem:1}
Let $H$ be a normal subgroup of $G$ defined over $\Q$ such that the $\R$-group $H_\R$ is compact.
Then $H\subset C$  (where $C$ denotes the center of $G$).
\end{lemma}

\begin{proof}
By assertion (b) of the Proposition we have $H_\R\subset K$, where $K$ is the centralizer of $\im\,\varphi$ in $G_\R$.
Therefore, $\im\,\varphi$ is contained in the centralizer (defined over $\Q$) $\sZ(H)$ of $H$ in $G$.
Then it follows from the definition of the Hodge group that $G\subset\sZ(H)$. Hence $H\subset C$.
\end{proof}

\begin{lemma} \label{lem:2}
Consider the natural representation  $\rho$ of the group $G'_\R$ in the vector space $V_\R$.
Denote by $r_\rho$ the number of pairwise nonequivalent summands in the decomposition of $\rho$
into a direct sum of $\R$-irreducible representations.
Then $r_\rho=\dim_\Q F_0$.
\end{lemma}

\begin{proof}
We set $\AA=\End^\circ A\otimes_\Q \R$ and write the decomposition
$$
\AA_1+\dots+\AA_{r_\AA}.
$$
of the semisimple $\R$-algebra $\AA$ into a sum of simple $\R$-algebras.
By the assertion (d) of the Proposition we have $r_\rho=r_\AA$.
Furthermore, it is known (see \cite[Section 21]{Mumford-book}) that the Rosati involution
acts on the center $F_i$ of the algebra $\AA_i$ trivially if $F_i=\R$,
and as the complex conjugation if $F_i=\C$.
It follows that
$$F_0\otimes_\Q \R=\R+\dots +\R$$
($r_\AA$ summands)
whence $\dim_\Q F_0=r_\AA$.
Thus $\dim_\Q F_0=r_\AA=r_\rho$.
\end{proof}

\begin{proof}[Proof of the theorem]
Denote by $r_\nc$ the number of noncompact groups in the decomposition
$$
G'_\R=G_1\cdot G_2 \cdots\cdot  G_N
$$
of the group $G'_\R$ in an almost direct product of simple $\R$-groups.
It is known from results of Satake \cite[Theorem 2]{Satake}
that for each $\R$-irreducible representation $\rho'$
in the decomposition of the representation $\rho$  into a direct sum of $\R$-irreducibles,
there exist not more that one noncompact group $G_i$ $(1\le i\le N)$
such that the restriction $\rho'|_{G_i}$ is nontrivial.
Therefore, $r_\nc\le r_\rho$.
Taking in account Lemma 2, we obtain that $r_\nc\le\dim_\Q F_0$\, .

Further, since $G'_\R$ is of Hermitian type,
we see that all the groups $G_1,\dots,G_N$ are of Hermitian type as well,
and hence they are absolutely simple.
Consider the action of the Galois group $\Gal(\Qbar/\Q)$
on the set of the simple factors $G_1,\dots,G_N$.
The orbits of the Galois group bijectively correspond to $\Q$-simple normal subgroups of $G'$.
By Lemma 1 each orbit contains at least one noncompact group $G_i$.
Thus the number of orbits, i.e., the number $r$ of $\Q$-simple normal subgroups of $G'$,
does not exceed the number $r_\nc$ of noncompact groups among $G_1,\dots,G_N$.
We obtain that $r\le r_\nc\le\dim_\Q F_0$, which completes the proof of the theorem.
\end{proof}

\begin{corollary}
Assume that $\End^\circ A=\Q$.
Write the decomposition
$$
\rho_\C=\rho_1\otimes\dots\otimes\rho_N
$$
of the irreducible representation $\rho_\C$ of the semisimple group $G_\C$
in the vector space $V_\C$ into a tensor product of irreducible representations
$\rho_i$ of the  universal coverings $\widetilde{G}_{i\C}$ $(i=1,\dots, N)$
of the simple factors $G_{i\C}$ of $G_\C$.
Then each of the representations $\rho_i$ respects a nondegenerate
skew-symmetric bilinear form, and the number $N$ is odd.
\end{corollary}

\begin{proof}
Indeed, by Corollary 1 the Galois group permutes transitively
the groups $G_{i\C}$ and the representations $\rho_i$.
Now Corollary 2 follows from assertion (e) of the Proposition.
\end{proof}

\hfill {\Small Submitted on November 17, 1980.}

\end{document}